\documentclass{amsart}
\usepackage{amsfonts,amssymb,amscd,amsmath,enumerate,verbatim,calc}
\usepackage{xy}

\newcommand{\CM}{Cohen-Macaulay}

\newcommand{\wrt}{with respect to}

\newcommand{\n}{\mathfrak{n} }
\newcommand{\m}{\mathfrak{m} }
\newcommand{\M}{\mathfrak{M} }

\newcommand{\R}{\mathcal{R} }
\newcommand{\Z}{\mathbb{Z} }
\newcommand{\Fc}{\mathcal{F} }
\newcommand{\Gc}{\mathcal{G}}
\newcommand{\rt}{\rightarrow}

\newcommand{\ov}{\overline}

\newcommand{\height}{\operatorname{height}}
\newcommand{\grade}{\operatorname{grade}}
\newcommand{\depth}{\operatorname{depth}}

\theoremstyle{plain}

\newtheorem{theorem}{Theorem}[section]
\newtheorem{corollary}[theorem]{Corollary}
\newtheorem{lemma}[theorem]{Lemma}
\newtheorem{proposition}[theorem]{Proposition}

\theoremstyle{definition}

\newtheorem{s}[theorem]{}
\newtheorem{remark}[theorem]{Remark}

\theoremstyle{remark}

\begin{document}

\title[Integral closure filtration]{On asymptotic depth of integral closure filtration and an application }
\author{Tony J. Puthenpurakal}
\date{\today}
\address{Department of Mathematics, Indian Institute of Technology Bombay, Powai, Mumbai 400 076, India}
\email{tputhen@math.iitb.ac.in}

 \begin{abstract}
Let $(A,\m)$ be an analytically unramified formally equidimensional Noetherian local ring with $\depth A \geq 2$. Let $I$ be an $\m$-primary ideal and set $I^*$ to be the integral closure of $I$. Set $G^*(I) = \bigoplus_{n\geq 0} (I^n)^*/(I^{n+1})^*$ be the associated graded ring of the integral closure filtration of $I$. We prove that $\depth G^*(I^n) \geq 2$ for all $n \gg 0$. As an application we prove that if $A$ is also an excellent normal domain containing an algebraically closed field isomorphic to $A/\m$ then there exists $s_0$ such that for all $s \geq s_0$ and $J$ is an integrally closed ideal \emph{strictly} containing $(\m^s)^*$  then we have a strict inequality $\mu(J) < \mu((\m^s)^*)$ (here
$\mu(J)$ is the number of minimal generators of $J$). 
\end{abstract}
 \maketitle
\section{introduction}

\s \label{setup} \emph{Setup:} Throughout $(A,\m)$ is a Noetherian local ring and $I$ is an $\m$-primary ideal Throughout we consider \emph{multiplicative} $I$-stable filtration of ideals \\ $\Fc = \{I_n \}_{n \geq 0}$, i.e., we assume
\begin{enumerate}
\item
$I_0 = A$ and $I_{n+1} \subseteq I_n$ for all $n \geq 0$
\item
$I_1 \neq A$ and  $I \subseteq I_1$.
\item
$I_n I_m \subseteq I_{n+m}$ for all $n, m \geq 0$.
\item 
$I I_n = I_{n+1}$ for all $n \gg 0$
\end{enumerate}
Let $\R(I) = \bigoplus_{n \geq 0} I^n t^n$ be the Rees-algebra of $A$ \wrt \ $I$.
Set $\R(\Fc) = \bigoplus_{n \geq 0} I_n t^n$ be the Rees-algebra of $\Fc$.
By our assumption $\R(\Fc)$ is a finite extension of $\R(I)$ in $A[t]$.
Also set $G(\Fc) = \bigoplus_{n \geq 0} I_n/ I_{n+1}$ the associated graded ring of $\Fc$. 
If $\Gc = \{ (I_n)^* \}$ the integral closure filtration of $I$ then we set 
$G^*(I) = G(\Gc)$. We note that if $A$ is analytically unramified (i.e., $\widehat{A}$ is reduced) then the integral closure filtration of $I$ is a multiplicative $I$-stable filtration of ideals.

The following result is definitely known. We sketch a proof for the convenience of the reader.
\begin{proposition}
(with hypotheses as in \ref{setup}).  Further assume that $A$ is analytically unramified. We have $I_n \subseteq (I^n)^*$ for all $n \geq 1$.
\end{proposition}
\begin{proof}[Sketch of a proof]
The integral closure of $\R(I)$ in $A[t]$ is the Rees-algebra of the integral closure filtration of $I$. By our assumption $\R(\Fc)$ is a finite extension of $\R(I)$ in $A[t]$. The result follows.
\end{proof}
\s As $A$ is reduced we have that $\depth A > 0$. It follows that the ideals $(I^n)^*$ are all Ratliff-Rush. So $\depth G^*(I) > 0$. 
Recall a Noetherian local ring $A$ is called \textit{formally equidimensional}
if  $\dim \widehat{A}/P = \dim A$ for all minimal primes of $\widehat{A}$, the completion of $A$ \wrt \ $\m$. In the literature, the local rings
with this property are also called quasi-unmixed.
Our main result is
\begin{theorem}\label{main}
Let $(A,\m)$ be an analytically unramified formally equidimensional Noetherian local ring with $\depth A \geq 2$. Let $I$ be an $\m$-primary ideal and set $I^*$ to be the integral closure of $I$.  Then $\depth G^*(I^n) \geq 2$ for all $n \gg 0$.
\end{theorem}

\begin{remark}
 We note that there exists $ n_0$ such that $(I^n)^*$ is a normal ideal for all $n \geq n_0$. If $A$ is \CM \ then it follows from a result of
 Huckaba and Huneke, \cite[3.1]{HH}, that $\depth G(((I^n)^*)^k) \geq 2$ for all $k \gg 0$. Even in this case it does not immediately follow 
 that $\depth G((I^n)^*) \geq 2$ for all $n \gg 0$.
\end{remark}

\textit{Application:}
Let $(A,\m)$ be a Noetherian local ring with $d=\dim A \ge 1$ and $I$ an
$\m$-primary ideal of $A$.  
The notion of $\m$-full ideals was introduced by 
D. Rees and J. Watanabe  (\cite{JW}) and they proved the \lq\lq Rees property" 
for $\m$-full ideals, namely, if $I$ is $\m$-full ideal and $J$ is an 
ideal containing $I$, then $\mu(J)\le \mu(I)$, where $\mu(I)=\ell_A(I/\m I)$ 
is the minimal number of generators of $I$. Also, they proved that integrally closed 
ideals are $\m$-full if $A$ is normal. 

\par
Suppose $\depth A >0$. 
Then $\widetilde{\m^n}$, 
the Ratliff-Rush closure of $\m^n$, is $\m$-full (\cite[Proposition 2.2]{AP}). 
Thus $\m^n$ is $\m$-full for sufficiently large $n$. 
\par 
Sometimes we need stronger property for $\mu(I)$ and we will call it 
\lq\lq Strong Rees property" (SRP for short). 

\par \vspace{2mm} \par \noindent 
{\bf Definition (Strong Rees Property).}
Let $I$ be an $\m$-primary ideal of $A$. 
Then we say that $I$ satisfies the {\it strong Rees property} 
if for every ideal $J \supsetneq I$, we have $\mu(J) < \mu(I)$.

In a recent paper the author with K. Watanabe and K. Yoshida proved the following result, \cite[3.2]{PWY}:
\begin{theorem}\label{PWY}
Let $(A,\m)$ be a Noetherian local ring. 
Assume that $\depth A \ge 2$ and 
$H_{\M}^1(G)$ has finite length, where $G=G(\m)= 
\oplus_{n\ge 0}\m^n/\m^{n+1}$.  
If $\m^{\ell}$ is Ratliff-Rush closed, then $\m^{\ell}$ has SRP.
\end{theorem}

For integrally closed ideals we have the following  weaker form of SRP.
\par \vspace{2mm} \par \noindent 
{\bf Definition ($SRP^*$ for integrally closed ideals ).}
Let $I$ be an $\m$-primary  integrally closed ideal of $A$. 
Then we say that $I$ satisfies the {\it *-strong Rees property for integrally closed ideals} 
if for every ideal $J \supsetneq I$ with $J$ \textit{integrally closed}, we have $\mu(J) < \mu(I)$.

Our  application of Theorem \ref{main} is the following:
\begin{theorem}\label{main-app}
Let $(A,\m)$ be an excellent normal domain  of dimension $d \geq 2$ with algebraically closed residue field. 
Then there exists $s_0$ such that if  $j \geq s_0$ and
 $I$ is an integrally closed ideal with  $I \varsupsetneq (\m^j)^* $;
 then $\mu(I) < \mu((\m^j)^*)$.
\end{theorem}

\emph{Technique used to prove our results:}\\
The main technique for this paper is to consider $L^{\Fc} =  \bigoplus_{n \geq 0} A/I_{n+1}$. This is a module over $\R(I)$, see \ref{basic}. When $\Fc$ is the $I$-adic filtration then this technique was developed in \cite{P1},\cite{P2}.
Although this module is not finitely generated as a $\R$-module it has the following good properties:
\begin{proposition}\label{tech}
 Assume $A$ is an analytically unramified local ring. Let $\Fc$ be a multiplicative $I$-stable filtration over an $\m$-primary
 ideal $I$. Let $\M$ be the maximal homogeneous ideal of $\R(I)$. Then
 \begin{enumerate}[\rm (1)]
  \item $H^0_\M(L^\Fc)_n = 0$ for all $n \gg 0$.
  \item
  Set $H^0_\M(L^\Fc)_n = I_{n+1}^\circ/I^{n+1}$. Then $I_{n+1}^\circ \subseteq (I_{n+1})^*$.
  \item If $\depth A \geq 2$ then we have
  \begin{enumerate}[\rm (a)]
   \item $H^1_\M(L^\Fc)_n = 0$ for all $n \gg 0$.
  \item $\ell (H^1_\M(L^\Fc)_n) < \infty $ for all $n \in \Z$.
  \item For all $n < 0$ we have
  $\ell (H^1_\M(L^\Fc)_{n-1})  \leq \ell (H^1_\M(L^\Fc)_{n})$.
  \end{enumerate}
\end{enumerate}
\end{proposition}

We now describe in brief the contents of this paper. In section two we describe a few preliminaries that we need.  In section three we describe our construction $L^\Fc$ and prove Proposition \ref{tech}. In section four we prove Theorem \ref{main}. Finally in section five we prove Theorem \ref{main-app}.

\section{preliminaries}
In this section we collect a few preliminaries that we need.
\s \emph{Some properties of multiplicative $I$-stable filtrations:}
\begin{enumerate}
\item
Let $\Fc = \{ I_n \}_{n \geq 0}$ be a multiplicative $I$-stable filtration and let $x \in I$. Set $B = A/(x)$, $\ov{I} = I/(x)$ and let
$\ov{\Fc} = \{ (I_n + (x))/(x)\}_{n \geq 0}$ be the quotient filtration of $\Fc$. Then $\ov{\Fc}$ is a multiplicative $\ov{I}$-stable filtration.
\item
If the residue field $k$ is infinite then we may choose $x$ which is both $I$ and $\Fc$-superficial.
\item
Set $\Fc^{<l>} = \{ I_{nl} \}_{n \geq 0}$. Then $\Fc^{<l>}$ is a multiplicative 
$I^l$-stable filtration. 
\end{enumerate}

\s\label{AtoA'} \textbf{Flat Base Change:} In our paper we do many flat changes of rings.
 The general set up we consider is
as follows:

 Let $\phi \colon (A,\m) \rt (A',\m')$ be a flat local ring homomorphism
 with $\m A' = \m'$. Set $I' = I A'$ and if
 $N$ is an $A$-module set $N' = N\otimes A'$. Set $k = A/\m$ and $k' = A'/\m'$.
 Let $\Fc = \{I_n \}_{n \geq 0}$ be a multiplicative $I$-stable filtration. Then $\Fc^\prime = \{I_n^\prime \}_{n \geq 0}$ is a multiplicative $I^\prime$-stable filtration.

 \textbf{Properties preserved during our flat base-changes:}

\begin{enumerate}[\rm (1)]
\item
$\ell_A(N) = \ell_{A'}(N')$.
\item
$\dim M = \dim M'$ and  $\grade(K,M) = \grade(KA',M')$ for any ideal $K$ of $A$.
\item
$\depth G(\Fc) = \depth G(\Fc^\prime)$.
\end{enumerate}

\textbf{Specific flat Base-changes:}

\begin{enumerate}[\rm (a)]
\item
$A' = A[X]_S$ where $S =  A[X]\setminus \m A[X]$.
The maximal ideal of $A'$ is $\n = \m A'$.
The residue
field of $A'$ is $k' = k(X)$. Notice that $k'$ is infinite.
\item
$A' =  \widehat{A}$  the completion of $A$ \wrt \ the maximal ideal.
\item
$A' = A[X_1,\ldots,X_n]_S$ where $S =  A[X_1,\ldots,X_n]\setminus \m A[X_1,\ldots,X_n]$.
The maximal ideal of $A'$ is $\n = \m A'$.
The residue
field of $A'$ is $l = k(X_1,\ldots,X_n)$. Notice that if $I$ is integrally closed then $I'$ is 
also integrally closed.
\item
When $\depth A \geq 2$ and $A$ is formally equidimensional, then
Ciuperc\u{a}   \cite[Corollary 2]{C} shows that in $A'$   (for $n = \mu(I)$ in (ii)) 
there exists a superficial element $y \in I'$ such that 
the $A'/(y)$ ideal $J = I'/(y)$ is also integrally closed (if $I$ is integrally closed).  We call $A'$ general extension 
of $A$ \wrt \ $I$. Also if $I = (a_1,\ldots, a_n)$ then we can choose $y$ to be $\sum_{i=1}^{n} a_i X_i$. We call the latter 
a generic element of $I'$.
We note that  Ciuperc\u{a} observes that $y$ is also superficial for the integral closure filtration of $I$, see \cite[Section 2.5]{C}.
\end{enumerate}

We will need the following result in the proof of Theorem \ref{main}. It is definitely known to experts. We give a proof 
for the convenience of the reader
\begin{lemma}\label{induct}
 Let $(A,\m)$ be an analytically unramified, formally equidimensional Noetherian local ring with $\depth A \geq 2$. Let
 $K$ be an $\m$-primaryideal of $A$. Then
 there exists a flat local homomorphism $\phi \colon A \rt A^\prime$ with $\m A^\prime = \m^\prime$ and that there exists $y \in K^\prime$ such that
 \begin{enumerate}[\rm (1)]
  \item the ring $A^\prime/(y)$ is analytically unramified.
  \item $y$ is superficial for both the $K^\prime$-adic filtration and the integral closure filtration of $K^\prime$.
 \item the ideal $(K^\prime)^* A^\prime/(y)$ is integrally closed in $A^\prime/(y)$.
 \item $ (K^n)^*A^\prime = ((K^\prime)^n)^*$ for all $n \geq 1$.
 \end{enumerate}
\end{lemma}
\begin{proof}
 We construct $A^\prime$ in two steps. First we complete $A$. Notice by assumption $\widehat{A}$ is equidimensional and reduced.
 Let $K = (a_1,\ldots, a_m)$. Then $K\widehat{A}$ is generated by images of $a_i$ in $\widehat{A}$. Set $A^\prime
 = \widehat{A}[X_1,\ldots,X_m]_{\widehat{\m}\widehat{A}[X_1,\ldots, X_m]}$ and $y = \sum_{i =1}^{m} a_iX_i$.
 Then as noted in \ref{AtoA'}(d)  we get that (2), (3) hold true. Also trivially (4) holds. 
 
 (1) Set $T = \widehat{A}[X_1,\ldots, X_m]$. 
 Let $P$ be a minimal prime ideal of $(y)$ in $A^\prime$. Then $\height P = 1$. There exists a prime ideal $Q$ of $T$ containing
 $(y)$ such that $\height Q = 1$ and $QA^\prime = P$. As $\depth A^\prime  \geq 2$ we get that there exists $a_i$ such that 
 $a_i \notin Q$. Say $a_m \notin Q$. Then $(T/(y))_Q$  is a localization of $\widehat{A}[X_1,\ldots, X_{m-1}]$. The latter ring is reduced.
 So $(T/(y))_Q$ is reduced. As $(A^\prime/(y))_P$ is a localization of $(T/(y))_Q$ it is also reduced. Thus $A^\prime/(y)$ is reduced.
 As it is also excellent we get that $A^\prime/(y)$ is analytically unramified.
\end{proof}

\section{$L^\Fc$}
In this section we define $L^\Fc$ and study few of its properties. We also prove 
 prove Proposition \ref{tech} which is the most important technical result of this paper.
 
\s \label{basic} Let $(A,\m)$ be a Noetherian local ring and let $I$ be an $\m$-primary ideal of $A$. Let $\Fc = \{ I_n \}_{n \geq 0}$ be an $I$-stable multiplicative  filtration of $A$. Set $L^\Fc = \bigoplus_{n \geq 0} A/I_{n+1}$.  Note that we have a short exact sequence of $\R(I)$-modules
\begin{equation}\label{defn}
0 \rt \R(\Fc) \rt A[t] \rt L^\Fc(-1) \rt 0.
\end{equation}
It follows that $L^\Fc$ is a $\R(I)$-module. Note that it is not finitely generated $\R(I)$-module.

Throughout we take local cohomology of $\R(I)$-modules \wrt \  $\M = \m \oplus \R(I)_+$, the maximal homogeneous ideal of $\R(I)$. Recall a graded $\R(I)$-module $V$ is said to be $*$-Artinian if every descending chain of graded submodules of $V$ stablizes. For instance if $E$ is a finitely generated graded $\R(I)$-module then 
for all $i \geq 0$ the $\R(I)$-modules  $H^i_\M(E)$ is $*$-Artinian.

\s \label{Artin-vanish} We will use the following well-known result regarding *-Artinian modules quite often:

Let $V$ be a *-Artinian $\R(I)$-module. Then
\begin{enumerate}[\rm (a)]
\item
$V_n = 0$ for all $n \gg 0$
\item
If $\psi \colon V(-1) \rt V$ is a monomorphism then $V = 0$.
\item
If $\phi \colon V \rt V(-1)$ is a monomorphism then $V = 0$.
\end{enumerate}

We begin with the following easy result
\begin{lemma}\label{a-vanish}
Let $\depth A = c \geq 1$. Then for $i = 0,\ldots, c-1$ we have $H^i_\M(L^\Fc)$ is $*$-Artinian. In particular
$H^i_\M(L^\Fc)_n = 0$ for $n \gg 0$.
\end{lemma}
\begin{proof}
We take local cohomology of short exact sequence (\ref{defn}) in 3.1  \wrt \ $\M$. We note that if $a_1,\ldots, a_c$ is an $A$-regular sequence then $a_1t^0,\ldots, a_ct^0 \in \M_0$ is an $A[t]$-regular sequence. This yields that $H^i_\M(L^\Fc(-1))$ to be  a $\R(I)$-sub-module of $H^{i+1}_\M(\R(\Fc))$ for $i = 0,\ldots, c-1$. The result follows.
\end{proof}
As an easy consequence we get the following
\begin{corollary}\label{zero}
Assume $A$ is analytically unramified. Set $L^*_I = \bigoplus_{n \geq 0} A/(I^{n+1})^*$. Then $H^0(L^*_I) = 0$.
\end{corollary}
\begin{proof}
We note that $\depth A > 0$. So by \ref{a-vanish} we get that $H^0_\M(L^\Fc)$ is $8$-Artinian. Also note that $(I^n)^*$ is Ratliff-Rush for all $n \geq 1$. In particular we have that $G^*(I)$ has positive depth, in particular $H^0(G^*(I)) = 0$.

We have a short exact sequence $0 \rt G^*(I) \rt  L^*_I \rt L^*_I(-1) \rt 0$. Taking local cohomology \wrt \ $\M$ we get an inclusion $H^0_\M(L^*_I) \hookrightarrow 
H^0_\M(L^*_I)(-1)$. The result follows from \ref{Artin-vanish}.
\end{proof}

We now give
\begin{proof}[Proof of Proposition \ref{tech}]
(1) This follows from \ref{a-vanish}.

(2)
As $A$ is analytically unramified the module $\R^*(I)/\R(\Fc) = \bigoplus_{n \geq 1} (I_n)^*/I_n$ is a finintely generated
$\R(I)$-module. We have a short exact sequence  of $\R(I)$-modules
$$0 \rt W(1) \rt L^\Fc \rt L^*_I \rt 0. $$
Taking local cohomology \wrt  \ $\M$ we get
\[
 0 \rt H^0_\M(W(1)) \rt H^0_\M(L^\Fc) \rt H^0_\M(L^*_I)
\]
But $H^0_\M((L^*_I) = 0$ by \ref{zero}. So we have
\[
 \frac{I_{n+1}^\circ}{I_{n+1}} = H^0_\M(L^\Fc)_n = H^0_\M(W(1))_n \subseteq W(1)_n  = \frac{(I^{n+1})^*}{I^{n+1}}.
\]
The result follows.

(3) (a) This follows from \ref{a-vanish}.

3(b) Let $x$ be both $\Fc$ and $I$-superficial. Then note we have a short exact sequence for all $n \geq 0$
\[
 0 \rt \frac{(I_{n+1} \colon x)}{I_n} \rt \frac{A}{I_n} \xrightarrow{\alpha_n} \frac{A}{I_{n+1}} \rt \frac{A}{(I_{n+1}, x)} \rt 0,
\]
where $\alpha_n( a+ I_n) = xa + I_{n+1}$.
Set $B = A/(x)$. Let $\ov{\Fc}$ be the quotient filtration of $\Fc$. Thus we have a short exact sequence of $\R(I)$-modules
\[
 0 \rt W \rt L^\Fc(-1) \xrightarrow{xt} L^\Fc \rt L^{\ov{\Fc}} \rt 0.
\]
As $x$ is $\Fc$-superficial we get that $W$ has finite length. So we have a short exact sequence
\begin{equation*}
  H^0_\M(L^{\ov{\Fc}})_n \rt H^1_\M(L^\Fc)_{n-1} \rt H^1_\M(L^\Fc)_n  \tag{$\dagger$}
\end{equation*}

Say $H^1_\M(L^\Fc)_n = 0$ for all $n \geq s$.
We note that $H^0_\M( L^{\ov{\Fc}})_n$
has finite length for all $n \geq 0$ and is zero for $n < 0$. Evaluating $(\dagger)$ at $n = s$ we get that 
$H^1_\M(L^\Fc)_{s-1}$
has finite length. Evaluating $(\dagger)$ at $n = s-1$ yields an exact sequence
\begin{equation*}
  H^0_\M(L^{\ov{\Fc}})_{s-1} \rt H^1_\M(L^\Fc)_{s-2} \rt H^1_\M(L^\Fc)_{s-1}
\end{equation*}
It follows that $H^1_\M(L^\Fc)_{s-2}$ has finite length. Iterating we get that $H^1_\M(L^\Fc)_n$ has finite length for all 
$n \leq s$.

3(c) This follows from $(\dagger)$ as $H^0_\M( L^{\ov{\Fc}})_n$ is zero for $n < 0$. 
\end{proof}

\section{Proof of Theorem \ref{main}}
In this section we give proof of Theorem \ref{main}. We also prove an additional result which will be useful in the proof
of Theorem \ref{main-app}.

\begin{proof}[Proof of Theorem \ref{main}]
 Let $\Fc = \{ (I^n)^* \}$ be the integral closure filtration of $I$. We note that for $r \geq 1$,
 $\Fc^{<r>}$ is the integral closure of filtration of $I^r$. Observe that
 \[
  \left(L^\Fc(-1)\right)^{<r>} = L^{\Fc^{<r>}}(-1)
 \]
 Also note that $\M^{<r>}$ is the maximal homogenous ideal of $\R(I^r)$.
As local cohomology commutes with the Veronese functor and as $H^1_\M(L^\Fc)_j = 0$ for $j \gg 0$ it follows that there exists
$r_0$ such that for $r \geq r_0$ we have $H^1_{\M^{<r>}}(L^{\Fc^{<r>}}(-1))_j = 0$ for all $j \geq 1$.

Fix $r \geq r_0$. Set $K = I^r$. We do the construction as in \ref{induct}. So we may assume that there exists $y \in K$ which is
superficial for both the $K$-adic filtration and the integral closure filtration of $K$. Furthermore $A/(y)$ is analytically 
unramified. Let $\Gc = \{ (K^n)^*\}$ be the integral closure filtration of $K$ and let $\ov{\Gc}$ be it's quotient filtration in 
$A/(y)$. Note we have an short exact sequence of $\R(K)$-modules
\[
 0 \rt L^\Gc(-1) \xrightarrow{yt} L^\Gc \rt L^{\ov{\Gc}} \rt 0.
\]
This induces a long exact sequence in cohomology. Note $H^0(L^\Gc) = 0$. Furthermore by construction $H^1(L^\Gc)_j = 0$ for 
$j \geq 0$. 
So for all $n \in \Z$ we have an exact sequence
\[
 0 \rt H^0(L^{\ov{\Gc}})_n \rt H^1(L^\Gc)_{n-1} 
\]
It follows that $H^0(L^{\ov{\Gc}})_n = 0$ for $n \geq 1$.
We note that as $A/(y)$ is analytically unramified we have by \ref{tech}(2)
\[
 H^0(L^{\ov{\Gc}})_0 = U/V, \quad \text{where} \ V = \ov{K^*} \ \text{and} \ U \subseteq  (\ov{K^*})^*, \ \text{ by \ref{tech}(2)}.
\]
But by our construction $\ov{K^*}$ is integrally closed in $A/(y).$ So $H^0(L^{\ov{\Gc}}) = 0$.
We also have a short exact sequence
\[
 0 \rt G(\ov{\Gc}) \rt L^{\ov{\Gc}} \rt L^{\ov{\Gc}}(-1) \rt 0.
\]
Taking cohomology we get $\depth G(\ov{\Gc}) \geq 1$. By Sally descent, 
\cite[2.2]{HM} we get $\depth G(\Gc) \geq 2$. The result follows.
\end{proof}

The following result is needed in the proof of Theorem \ref{main-app}.
\begin{lemma}\label{sun}
Let $(A,\m)$ be an analytically unramified formally equidimensional Noetherian local ring with $\depth A \geq 2$.
Let $I$ be an $\m$-primary ideal.
 Set $\Fc = \{ (I^n)^* \}_{n \geq 0}$ to be the integral closure filtration of $I$.
 Then $H^1(L^\Fc)$ has finite length.
\end{lemma}
\begin{proof}
 By Theorem \ref{main} there exists $r\geq 1$  such that $G^*(I^r)$ has depth atleast $r$. Let $\Gc$ be the integral closure filtration
 of $I^r$. So we have an exact sequence
 \[
  0 \rt G(\Gc) \rt L^\Gc \rt L^\Gc(-1) \rt 0.
 \]
 Taking cohomology we get an injective map $H^1(L^\Gc) \rt H^1(L^\Gc)(-1)$. By \ref{Artin-vanish} we get
 $H^1(L^\Gc) = 0$.
 
 As observed before 
 \[
  (L^\Fc(-1))^{<r>} = L^\Gc(-1).
 \]
As local cohomology commutes with the Veronese functor we get that \\
$H^1(L^\Fc)_{-r-1} = 0$. The result now follows from \ref{tech}(3).
\end{proof}

\section{Proof of Theorem \ref{main-app}}
Our proof is quite similar in spirit to proof of Theorem 1.6. However it is different in some places. So we are forced to give the whole proof. 
\begin{proof}[Proof of Theorem \ref{main-app}]
We first note that as $A$ is excellent and normal $\widehat{A}$ is also a normal domain of dimension $ d \geq 2$, see \cite[32.2]{Ma}. In particular $A$ is analytically unramified, formally equidimensional with $\depth A \geq 2$.
So all the techniques developed in the earlier sections are applicable.

The Rees algebra of the integral closure filtration of $\m$ is a finite module over $\R(\m)$. In particular there exists an $s_0$ such that for all $j \geq s_0$ 
we have $\m (\m^j)^* = (\m^{j+1})^*$. Choose  $j \geq s_0$.

\textit{Step-1} We we may assume $\lambda(I/(\m^j)^*) = 1$. \\
 By  \cite[2.1]{W} there exists a chain of integrally closed ideals
 \[
  I = I_0 \varsupsetneq I_1 \varsupsetneq I_2 \varsupsetneq \cdots \varsupsetneq I_{s_1} \varsupsetneq I_s = (\m^j)^*,
 \]
 with $\lambda(I_i/I_{i+1}) =1$ for $i = 0,\ldots,s-1$. As $I_j$ are integrally closed  and $A$ is normal they are $\m$-full. So we have 
 \[
  \mu(I) = \mu(I_0) \leq \mu(I_1) \leq \mu(I_2) \leq \cdots \leq \mu(I_{s-1}) \leq \mu((\m^j)^*).
 \]
Thus it suffices to prove $\mu(I_{s-1}) < \mu((\m^j)^*)$. Thus we may assume $\lambda(I/(\m^j)^*) = 1$.

\textit{Step-2} A consequence of assuming $\lambda(I/(\m^j)^*) = 1$. \\
Note $\m I \subseteq (\m^j)^*$.  Also note that 
$$(\m^{j+1})^* = \m (\m^j)^* \subseteq \m I.$$

We have an exact sequence
\[
 0 \rt \frac{(\m^j)^* \cap \m I}{(\m^{j+1})^*}  \rt \frac{(\m^j)^*}{(\m^{j+1})^*}  \rt \frac{I}{\m I} \rt \frac{I}{(\m^j)^* + \m I} \rt 0.
\]
We note that $(\m^j)^* + \m I = (\m^j)^*$ and  $(\m^j)^* \cap \m I = \m I$. Furthermore $\m (\m^j)^* = (\m^{j+1})^*$. Thus the above exact sequence yields an exact sequence
\[
 0 \rt \frac{\m I}{(\m^{j+1})^*}  \rt \frac{(\m^j)^*}{ \m (\m^{j})^*}  \rt \frac{I}{\m I} \rt \frac{I}{\m^j } \rt 0.
\]
Thus to prove our result it suffices to show $c = \lambda(\m I/(\m^{j+1})^*) \geq 2$.

To show $c \geq 2$ we have to consider some modules over the Rees algebra 
$\R(\m) = A[\m t]$ of $\m$. 

\textit{Step 3:} Some modules over the Rees algebra $\R(\m)$. \\
We first list some finitely generated $\R(\m)$-modules that are pertinent to our result.
\begin{enumerate}[\rm (1)]
\item $\R^* = \bigoplus_{i \geq 0}(\m^i)^*t^i$ is a finitely generated $\R(\m)$-module.
 \item  $\R^*_{\geq j} = \bigoplus_{i \geq j} (\m^i)^* t^i$ is an $\R(\m)$-submodule of $\R^*$. \\
 So $\R^*_{\geq j}(+ j) = \bigoplus_{i \geq 0} (\m^{j+i})^*t^i$ is a finitely generated $\R(\m)$-module.
 \item The filtration 
 $$I \supseteq \m I \supseteq \m^2 I \supseteq \cdots \supseteq \m ^i I \supseteq \m^{i+1} I \supseteq \cdots$$
 is $\m$-stable (note that it is not multiplicative). So $E = \bigoplus_{i \geq 0}\m ^i It^i$ is a finitely generated $\R(\m)$-module.
 \item
By our assumption on $j$ we get that   $\R^*_{\geq j}(+ j) \subseteq E$. So $C = E/\R^*_{\geq j}(+ j)$ is a fintely generated $\R$-module.
 Note
 \begin{align*}
  C &= \frac{I}{(\m^j)^*} \oplus \frac{\m I}{(\m^{j+1})^*} t \oplus \frac{\m^2 I}{(\m^{j+2})^*} t^2 \oplus \cdots 
  \frac{\m^i I}{(\m^{j+i})^*}t^i \oplus \cdots, \\
  &= \frac{I}{(\m^j)^*} \oplus \frac{\m I}{\m (\m^{j})^*} t \oplus \frac{\m^2 I}{\m^2(\m^{j})^*} t^2 \oplus \cdots 
  \frac{\m^i I}{\m^i(\m^{j})^*}t^i \oplus \cdots.
 \end{align*}
\end{enumerate}
We now list two NOT finitely generated modules over $\R(\m)$ which is of interest to us:

\begin{enumerate}[\rm (1)]
 \item $L = \bigoplus_{i \geq 0} A/(\m^{i+1})^*$. By \ref{basic} we get that $L$
 is a $\R(\m)$-module.
\item
We also have an exact sequence
\[
 0 \rt E \rt A[t] \rt V \rt 0,
\]
where $V =  \bigoplus_{i \geq 0} A/\m^{i}I$. It follows that $V$ is a $\R$-module.
\end{enumerate}

\textit{Step 4:} Some local cohomology computations. \\
Throughout we compute local cohomology \wrt \ $\M$, the maximal homogeneous ideal of $\R(\m)$.
As $\depth A \geq 2$ there exists  $x, y \in \m$ such that $x, y$ is an $A$-regular sequence. We note that $xt^0, yt^0 \in \M_0$
is an $A[t]$-regular sequence. So $H^0(A[t]) = H^1(A[t]) = 0$. It follows that $H^0(V) = H^1(E)$.

(1) \textit{Claim-1:} $H^0(V)$ has finite length as an $A$-module. 

It is well-known that  for all $l \geq 0$ we have $H^l(E)_n = 0 $ for $n \gg 0$.
As $H^0(V) = H^1(E)$  we get that $H^0(V)_n = 0$ for $n \gg 0$.

Also as $H^0(V) \subseteq V $ we have $H^0(V)_n = 0$ for $n < 0$. Furthermore
as $H^0(V)_n \subseteq V_n$ it follows that $H^0(V)_n$ has finite length for all $n$. So Claim-1 follows.

(2) By \ref{zero} we get $H^0(L) = 0$. By  \ref{sun}
we get that $ H^1(L)$ have finite length. Set $L(-1)_{\geq j} = \bigoplus_{i \geq j} A/\m^i t^i$.
We have an exact sequence
\[
 0 \rt L(-1)_{\geq j} \rt L(-1) \rt W \rt 0,
\]
where $W$ has finite length. It follows that $H^0(L(-1)_{ \geq j}) = 0$
and  $H^1(L(-1)_{ \geq j})$
 has finite length.
\end{proof}

\textit{Step 5:} $c = \lambda(\m I /(\m^{j+1})^*) > 1$.

Set $D = (L(-1)_{\geq j})(+ n) = \bigoplus_{i \geq 0} A/(\m^{i+j})^*$. 
Note we have a short exact sequence of $\R(\m)$-modules
\[
 0 \rt C \rt D \rt V \rt 0.
\]
Let $x$ be $A$-superficial \wrt \ $\m$. Then  it is also superficial \wrt \ to the
integral closure filtration of $\m$; see \cite[section 2.5]{C}. Notice that $((\m^{n+1})^* \colon x) = (\m^n)^*$ for all $n \geq 1$.

Set $u = xt \in \R_1$. Let $ f^D_u \colon  D(-1) \rt  D$  be multiplication by $u$. Notice $\ker f^D_u = 0$.  It follows that $ \ker(f^C_u) = 0$.

Notice $ \ker u$ is $D$-regular. So $u$ is $C$-regular.

\textit{Suppose if possible $c = 1$}. We note that $C_0 = I/(\m^j)^*$ and $C_1 = \m I /\m (\m^{j })^*$.  Note As $\ker(f^C_u)= 0$
we get the map $C_0 \xrightarrow{u} C_1$ is an isomorphism. So we get
$\m I = x I + \m(\m^{j})^*$. It follows that for all $i \geq 1$ we have
\[
 \m^{i+1}I = x \m^i I +\m^i (\m^{j})^*.
\]
Thus we have that
\[
 0 \rt  \rt C(-1) \xrightarrow{u} C \rt I/\m^j \rt 0.
\]
(here  $I/\m^j$ is concentrated in degree zero). It follows that $\dim_\R C = 1$.

As $\dim_\R C = 1$ we get that $H^1(C)$ is NOT finitely generated as an $A$-module.  However the exact sequence
$0 \rt C \rt D \rt V \rt 0$ yields an exact sequence $H^0(V) \rt H^1(C) \rt H^1(D)$ which implies that $H^1(C)$ has finite 
length.

Thus our assumption $c = 1$ is not possible. So $c \geq 2$ and this proves our result.


\begin{thebibliography}{20}

\bibitem{AP}
J.~Asadollahi and T.~J.~Puthenpurakal, 
\emph{An analogue of a theorem due to Levin and Vasconcelos}, 
Commutative algebra and algebraic geometry, 9--15,
Contemp. Math., {\bf 390}, Amer. Math. Soc., Providence, RI, 2005. 

\bibitem{C}
C. Ciuperc\u{a},  
\emph{Integral closure and generic elements},
 J. Algebra 328 (2011), 122–-131

\bibitem{HH}
S.~Huckaba and C.~Huneke, 
\emph{Normal ideals in regular rings}, 
J. Reine Angew. Math. 510 (1999), 63-82. 

\bibitem{HM}
S.~Huckaba and T.~Marley, 
\emph{Hilbert coefficients and the depths of associated graded rings},
 J. London Math. Soc. 56 (2) (1997) 64–-76.
 
 \bibitem{Ma}
H.~Matsumura, \emph{Commutative ring theory}, Cambridge Studies in Advanced
  Mathematics, vol.~8, Cambridge University Press, Cambridge, 1986. 

\bibitem{P1}
T.~J.~Puthenpurakal, 
\emph{Ratliff-Rush filtration, regularity and depth of higher associated graded modules. I},
J. Pure Appl. Algebra {\bf 208}  (2007), no.1, 159--176. 

\bibitem{P2}
T.~J.~Puthenpurakal, 
\emph{Ratliff-Rush filtration, regularity and depth of higher associated graded modules. Part II},
 J. Pure Appl. Algebra \textbf{221} (2017), no. 3, 611--631.
 
 \bibitem{PWY}
T.~J.~Puthenpurakal, K. ~Watanabe and K.~Yoshida,
\emph{The strong Rees property of powers of 
the maximal ideal and Takahashi-Dao's question}.
Preprint.  arXiv:1708.06090 




\bibitem{JW}
 J.~Watanabe, 
\emph{$\m$-full ideals}, 
  Nagoya Math. ~J. {\bf 106} (1987), 101--111.  

\bibitem{W}
 K.-i. ~Watanabe,
 \emph{Chains of integrally closed ideals},
  Commutative algebra (Grenoble/Lyon, 2001), 353--358,
  Contemp. Math., {\bf 331}, Amer. Math. Soc., Providence, RI, 2003.
\end{thebibliography}
\end{document}